\documentclass[12pt,reqno]{amsart}
\usepackage{fullpage}
\usepackage[top=1.4in, bottom=1.38in, left=1.38in, right=1.38in]{geometry}

\usepackage{graphicx}
\usepackage{amsmath,amsopn,amssymb,amsfonts,amsthm}
\usepackage{mathtools}
\usepackage{color}
\usepackage{ulem}

\usepackage{enumitem}
\usepackage[framemethod=TikZ]{mdframed}
\usepackage{bbm}
\usepackage{mathrsfs}
\usepackage{booktabs}
\usepackage{caption}
\usepackage{bm}
\usepackage{tensor}
\usepackage{cleveref}
\usepackage{cancel}

\renewcommand{\H}{\mathbb{H}}

\newcommand{\SL}{\mathrm{SL}}

\newcommand{\N}{\mathbb N}
\newcommand{\C}{\mathbb C}

\theoremstyle{plain}
\newtheorem{thm}{Theorem}[section]

\newtheorem{lem}[thm]{Lemma}

\theoremstyle{definition}

\newtheorem{defn}[thm]{Definition}

\numberwithin{equation}{section}

\newcommand{\pa}[2]{\left(\frac{#1}{#2}\right)}

\def\k{\kappa}

\def\n{\nu}

\def\k{\kappa}

\def\n{\nu}

\newcommand{\re}{{\rm Re}}
\newcommand{\im}{{\rm Im}}

\newcommand{\R}{\mathbb R}
\newcommand{\Z}{\mathbb Z}

\newcommand{\Log}{\operatorname{Log}}

\renewcommand{\binom}[2]{\left(\begin{smallmatrix}#1\\\\#2\end{smallmatrix}\right)}

\setlist[itemize]{noitemsep, topsep=0pt}

\allowdisplaybreaks

\makeatletter
\newcommand{\vast}{\bBigg@{3}}
\newcommand{\Vast}{\bBigg@{5}}
\makeatother

\DeclareMathOperator{\srp}{srp}

\title[Modularity of moments of reciprocal sums for distinct partitions]{Modularity of moments of reciprocal sums for partitions into distinct parts}

\author[K. Bringmann]{Kathrin Bringmann}
\address{Department of Mathematics and Computer Science\\Division of Mathematics\\University of Cologne\\ Weyertal 86-90 \\ 50931 Cologne \\Germany}
\email{kbringma@math.uni-koeln.de}

\author[B. Kim]{Byungchan Kim}
\address{School of Natural Sciences, Seoul National University of Science and Technology, 232 Gongneung-ro, Nowon-gu, Seoul, 01811, Republic of Korea}
\email{bkim4@seoultech.ac.kr}

\author[E. Kim]{Eunmi Kim}
\address{Institute of Mathematical Sciences, Ewha Womans University,  52 Ewhayeodae-gil, Seodaemun-gu, Seoul 03760, Republic of Korea}
\email{ekim67@ewha.ac.kr; eunmi.kim67@gmail.com}

\makeatletter
\@namedef{subjclassname@2020}{\textup{2020} Mathematics Subject Classification}
\makeatother

\subjclass[2020]{11F20, 11P82, 11P83}
\keywords{Eichler integrals, Eisenstein Series, integer partitions, raising operator, Maass forms, sesquiharmonic Maass forms, sums of reciprocal parts}

\begin{document}

\begin{abstract}
In this paper, we determine modularity properties of the generating function of $s_k(n)$ which sums $k$-th power of reciprocals of parts throughout all of the partitions of $n$ into distinct parts. In particular, we show that the generating function for $s_k (n)$ is related to Maass Eisenstein series and sesquiharmonic Maass forms.
\end{abstract}

\maketitle

\section{Introduction and statements of results}

A {\it partition} of $n\in\N_0$ is a non-increasing sequence of positive integers $\lambda_1\ge\lambda_2\ge\dots\ge\lambda_\ell$ such that the {\it parts} $\lambda_j$ sum up to $n$. We denote by $p(n)$ the number of partitions of $n$. Its generating function is, by Euler,
\begin{equation*}
	P(q):=\sum_{n\ge0} p(n)q^n = \frac{1}{(q;q)_\infty},
\end{equation*}
where, for $a\in\C$, $n\in\N_0\cup\{\infty\}$, we let $(a;q)_n:=\prod_{j=0}^{n-1}(1-aq^j)$. Note that $P(q)$ is (essentially) a modular form. This modularity has many applications including Hardy--Ramanujan's asymptotic formula \cite{HR}, Rademacher's exact formula \cite{Rad}, and Ramanujan's partition congruences \cite{Ra}.  Graham \cite{Graham} studied partitions with conditions on the reciprocal of parts.  We let $\srp (\lambda):= \sum_{j=1}^{\ell(\lambda)} \frac{1}{\lambda_j}$, where $\ell(\lambda)$ is the number of parts of the partition $\lambda$ and we define $\mathcal{D}_n$ to be the set of partitions of $n$ into distinct parts. Graham \cite[Theorem 1]{Graham} proved that there exists a partition $\lambda \in \mathcal{D}_n$ with $\srp(\lambda)=1$ if $n \ge 78$. To understand how $\srp(\lambda)$ is distributed along $\mathcal{D}_n$, the last two authors \cite{KK} introduced the counting functions\footnote{In \cite{KK}, $s_1(n)$ and $s_2(n)$ were denoted by $s(n)$ and $ss(n)$.} $s_1(n)$ and $s_2(n)$ as
\[
	s_1(n) := \sum_{\lambda \in \mathcal{D}_n}  \srp (\lambda)
	\quad\text{and}\quad
	s_2(n) := \sum_{\lambda \in \mathcal{D}_n} \srp^2(\lambda),
\]
which are the first and the second moments of the reciprocal sums of partitions.

In \cite[Theorems 1.1 and 1.3]{BKK}, the authors of this paper obtained asymptotic formulas for $s_1(n)$ and $s_2(n)$ with polynomial error by analyzing the behavior of their generating functions near roots of unity using the Circle Method. Along the way, we found that these functions introduce extra terms if we apply modular transformations. 

More generally, following \cite{BKK}, let $s_k(n)$ be the {\it$k$-th moment of the reciprocal sums of partitions} for $k\in\N_0$, i.e.,
\[
	s_k(n) := \sum_{\lambda \in \mathcal{D}_n} \srp^k (\lambda) = \sum_{\lambda \in \mathcal{D}_n} \left( \sum_{j=1}^{\ell(\lambda)} \frac{1}{\lambda_j}\right)^k.
\]
Then the generating function for $s_k(n)$ is
\[
	\sum_{n \geq 1} s_k (n) q^n
	= \left[ \left(\zeta \frac{d}{d\zeta} \right)^k \prod_{r\ge1} \left( 1+ \zeta^\frac1rq^r \right) \right]_{\zeta=1}.
\]
Define
\begin{equation*}
	g_k (q) := \left[ \left(\zeta \frac{d}{d\zeta} \right)^k \sum_{r\ge1} \Log \left(1+\zeta^{\frac1r} q^r\right) \right]_{\zeta=1},
\end{equation*}
where $\Log$ denotes the principal branch of the logarithm. Note that
\[
	\frac{1}{(-q;q)_{\infty}} \sum_{n \geq 1} s_k (n) q^n \in \mathbb{Z}\left[g_1(q), g_2(q), \ldots, g_k (q)\right].
\]
For instance, we have the following generating functions for the first and the second moments, as in Subsection 4.1 of\footnote{The $q$-series $g_1 (q)$ could be interpreted as a distinct partition analogue of $q$-bracket as it is $\frac{ \sum_{\lambda \in \mathcal{D} } {\rm srp} (\lambda) q^{|\lambda|} }{ \sum_{\lambda \in \mathcal{D}} q^{| \lambda |} }$, which is similar to the spin $q$-bracket defined in \cite[Subsection 2.1]{IS}.} \cite{KK}:
\begin{align}
	\frac1{ (-q;q)_\infty } \sum_{n\ge1} s_1(n) q^n &= g_1(q) = \sum_{n\ge1}\frac{q^n}{n\left(1+q^n\right)},\label{g1}\\
	\frac1{ (-q;q)_\infty } \sum_{n\ge1} s_2(n) q^n &= g_1^2(q) + g_2(q) = \left( \sum_{n\ge1}\frac{q^n}{n\left(1+q^n\right)}\right)^2+ \sum_{n\ge1}\frac{q^n}{n^2\left(1+q^n\right)^2}.\nonumber
\end{align}
Throughout the paper, we write $q=e^{2\pi i \tau}$ with $\tau = u + iv \in \mathbb{H}$, the complex upper half-plane.
Note that $g_1(q)$ and $g_2(q)$ are reminiscent of the {\it weight $2k$ Eisenstein series} 
\[
	E_{2k}(\tau) := 1-\frac{4k}{B_{2k}}\sum_{n\ge1} \sigma_{2k-1}(n) q^n = 1 - \frac{4k}{B_{2k}} \sum_{n \geq 1} \frac{n^{2k-1} q^n}{1-q^n},
\]
where $B_\ell$ denotes the {\it $\ell$-th Bernoulli number} and $\sigma_j (n):= \sum_{{d\mid n}}{d^j}$ in the {\it $j$-th divisor sum}. Based on this observation and the nice asymptotic behaviors and the modular-like transformations of $g_1(q)$ and $g_2(q)$ \cite{BKK}, the authors asked whether $g_k(q)$ can be ``completed'' to a modular object. In this paper, we show that this is indeed the case.

First, we consider $g_1(q)$ defined in \eqref{g1} as it behaves differently than $g_k(q)$ for $k\ge2$. Define
\begin{equation*}
	\widehat g_1(\tau) := 2\log(2) - \gamma + \frac{6\zeta'(2)}{\pi^2} - \frac\pi2 v + \frac{\log(v)}2 + g_1(q) + g_1(\overline q),
\end{equation*}
where $\zeta (s)$ denotes the usual Riemann zeta function. Note that $g_1(q)$ is basically the ``holomorphic part'' of $\widehat g_1(\tau)$. Our first result shows that $\widehat g_1$ is a so-called sesquiharmonic Maass form. Recall that harmonic Maass forms of weight $k$ are annihilated by $\Delta_k = -\xi_{2-k}\circ\xi_k$, whereas sesquiharmonic Maass forms are annihilated by $-\xi_k\circ\xi_{2-k}\circ\xi_k$ (see \Cref{S:Prelim} for details). Here the {\it shadow operator} is defined as 
\begin{equation*}
	\xi_k := 2iv^k \overline{\frac\partial{\partial\overline\tau}}.
\end{equation*}
We let
\begin{equation*}
	\widehat E_2(\tau) := 1 - 24\sum_{n\ge1} \sigma_1(n) q^n - \frac3{\pi v}
\end{equation*}
be the {\it weight two non-holomorphic Eisenstein series}.

\begin{thm}\label{thm:g1}
	The function $\widehat g_1$ is a sesquiharmonic Maass form of weight zero on $\Gamma_0(2)$. Moreover, we have
	\begin{equation*}
		\xi_0\left(\widehat g_1(\tau)\right) = \frac\pi6\left(\widehat E_2(\tau) - 4\widehat E_2(2\tau)\right).
	\end{equation*}
\end{thm}

Let $E(\tau;s)$ denote the classical Maass Eisenstein series defined in \eqref{Eis}.  We also realize $\widehat g_1$ as a limit of a Maass Eisenstein series.
\begin{thm}\label{thm:g_1 limit}
	We have
	\begin{equation*}
		\widehat g_1(\tau) = -\frac{\pi}{6} \lim_{s\to1} \left(2E(2\tau;s)-E(\tau;s)-\frac{3}{\pi(s-1)}\right).
	\end{equation*}
\end{thm}

We next consider $g_k(q)$ for $k\ge2$. We define for $k\in\mathbb{N}$ with $k\ge2$
\begin{equation*}
	\widehat g_k(\tau) := \frac{(-1)^{k+1} (4\pi)^k(k-1)!B_{2k}}{2(2k)!}\left(E(\tau;k)-2^k E(2\tau;k)\right).
\end{equation*}

\begin{thm}\label{thm:gk}
	Let $k\ge2$ be an integer. Then $g_k(q)$ is the holomorphic part of{\hspace{.07cm}}\footnote{Throughout by {\it holomorphic part} we mean the contribution that is a $q$-series - see also the proof of Theorem 1.2.} $\widehat g_k(\tau)$. Moreover, $\widehat g_k$ is a Maass form on $\Gamma_0(2)$ with eigenvalue $k(1-k)$.
\end{thm}

The paper is organized as follows. In Section 2, we recall basic facts on (weak) Maass forms, the raising operator, sesquiharmonic Maass forms, Maass Eisenstein series, and Eichler integrals. In Section 3, we prove Theorems \ref{thm:g1} and \ref{thm:g_1 limit}. The proof of Theorem  \ref{thm:gk} is given in Section 4.  In Section 5, we give an exemplary explicit calculation of $\widehat{g}_2 (\tau)$. We end in Section 6 with questions for further research.

\section*{Acknowledgments}
The authors thank Walter Bridges, Larry Rolen, and Jan-Willem van Ittersum for useful comments on an earlier version of this paper. In particular they thank Andreas Mono for pointing the authors to the Kronecker limit formula which lead to the proof of \Cref{thm:g_1 limit}.  The first author has received funding from the European Research Council (ERC) under the European Union's Horizon 2020 research and innovation programme (grant agreement No. 101001179). The third author was supported by the Basic Science Research Program through the National Research Foundation of Korea (NRF) funded by the Ministry of Education (RS-2023-00244423).

\section{Preliminaries}\label{S:Prelim}

\subsection{Weak Maass forms}

\begin{defn}
	A real-analytic function $f\colon\H\to\C$ is a {\it weak Maass form of weight $k\in\Z$ and eigenvalue $\lambda$ on} $\Gamma_0(N)$ if the following conditions hold: 
	\begin{enumerate}[wide,labelwidth=!,labelindent=0pt]
		\item We have, for $\begin{psmallmatrix}a&b\\c&d\end{psmallmatrix} \in \Gamma_0(N)$,
		\[
			f\left(\frac{a\tau+b}{c\tau+d}\right) = (c\tau+d)^k f(\tau).
		\]
		
		\item We have
		\[
			\Delta_k(f) = \lambda f,
		\]
		where the {\it weight $k$ Laplace operator} is defined as
		\[
			\Delta_k := -v^2\left(\frac{\partial^2}{\partial u^2} + \frac{\partial^2}{\partial v^2}\right) + ikv\left(\frac{\partial}{\partial u} + i\frac{\partial}{\partial v}\right).
		\]
		Note that
		\[
			\Delta_k = -\xi_{2-k}\circ\xi_k.
		\]
		
		\item The function $f$ has at most linear exponential growth at the cusps of $\Gamma_0(N)$.
	\end{enumerate}
	If $\lambda=0$, then $f$ is called a {\it harmonic Maass form}. 
\end{defn}

Harmonic Maass forms naturally split into a holomorphic and a non-holomorphic part. To be more precise, if $f$ is a harmonic Maass form of weight $k\in\frac12\N \setminus \{1\}$, then $f$ can be written as
\begin{equation*}
	f(\tau) = \sum_{n\gg-\infty} c_f^+(n) q^n + c_f^-(0) v^{1-k} + \sum_{\substack{n\ll \infty\\n \neq 0}} c_f^-(n) \Gamma(1-k,-4\pi nv) q^n,
\end{equation*}
where $\Gamma(s,y):=\int_y^\infty t^{s-1} e^{-t} dt$ is the {\it incomplete gamma function}. The first sum is the {\it holomorphic part} of $f$ and the remaining parts are its {\it non-holomorphic part}.

\subsection{The raising operator}
Define the {\it raising operator}
\begin{equation*}
	R_\kappa := 2i\frac\partial{\partial\tau} + \frac\kappa v.
\end{equation*}
This operator maps harmonic Maass forms of weight $\kappa$ to Maass forms of weight $\kappa+2$ and eigenvalue $\kappa$ (see \cite[Lemma 5.2]{BFOR}). Moreover, we require the iterated raising operator. To be more precise, for $n\in\N$, we define
\begin{equation*}
	R_\k^n := R_{\k+2(n-1)}\circ...\circ R_{\k}.
\end{equation*}
According to \cite[Lemma 5.4]{BFOR}, 
\begin{equation}\label{E:RD}
	R_{2-2k}^n = \sum_{r=0}^n (-1)^r \binom nr (2-2k+r)_{n-r} v^{r-n} (4\pi)^r D^r,
\end{equation}
where the {\it rising factorial} is defined as $(a)_n\coloneqq a(a+1)\cdots(a+n-1)$ and $D:=\frac1{2\pi i} \frac{\partial}{\partial \tau}$.

\subsection{Sesquiharmonic Maass forms}
\begin{defn}
	A real-analytic function $f:\H\to\C$ is called a {\it sesquiharmonic Maass form of weight $k$ on $\Gamma_0(N)$} if it satisfies (1) and (3) of Definition 2.1 and instead of (2) we have
	\[
		\Delta_{k,2}(f) = 0,
	\]
	where
	\[
		\Delta_{k,2} := -\xi_k\circ\xi_{2-k}\circ\xi_k.
	\]
\end{defn}

\subsection{Maass Eisenstein series}
For $s\in\C$ with $\re(s)>1$, we define the classical {\it Maass Eisenstein series}
\begin{equation}
	E(\tau; s) := \sum_{\gamma\in\Gamma_\infty\setminus \SL_2(\Z)} \im(\gamma\tau)^s, \label{Eis}
\end{equation}
where $\Gamma_\infty :=\{\pm\begin{psmallmatrix}1&n\\0&1\end{psmallmatrix} : n\in\Z\}$. The Maass Eisenstein series are Maass forms of weight $0$ on $\SL_2(\Z)$ with eigenvalue $s(1-s)$ under $\Delta_0$. In the following lemma, we state the Fourier expansion of $E(\tau; s)$.

\begin{lem}[{\cite[Theorem 3.9 iii)]{BFOR}}]\label{L:BFOR}
	Let $s\in\C$ with $\re(s)>1$. The Fourier expansion of $E(\tau;s)$ is given by
	\begin{equation*}
		E(\tau; s) = v^s + \phi(s) v^{1-s} + 2\sqrt v \sum_{n\in\Z\setminus\{0\}} \phi(n,s) K_{s-\frac12} (2\pi |n| v) e^{2\pi inu},
	\end{equation*}
	where $K_\kappa(x)$ is the $K$-Bessel function of order $\kappa$, 
	\begin{equation*}
		\phi(s) := \sqrt\pi \frac{\Gamma\!\left(s-\frac12\right)\zeta(2s-1)}{\Gamma(s)\zeta(2s)} \quad\text{and}\quad \phi(n,s) := \frac{\pi^s}{\Gamma(s)\zeta(2s)} \sum_{ab=|n|} \pa ab ^{s-\frac12}.
	\end{equation*}
\end{lem}
The Kronecker limit formula (see Theorem 1 of \cite{Si}) yields that
\begin{equation}\label{Kr} 
		\lim_{s\to1} \left(E(\tau;s)-\frac3{\pi(s-1)}\right) = \frac6\pi \left(\gamma-\log(2)-\log\!\left(\sqrt{v}|\eta(\tau)|^2\right)\right) - \frac{36}{\pi^3}\zeta'(2),
\end{equation}
where the {\it Dedekind $\eta$-function} is defined as
\begin{equation*}
	\eta(\tau) := q^{\frac{1}{24}} \prod_{n\ge1} \left(1-q^n\right) = \frac{q^\frac{1}{24}}{P(q)}.
\end{equation*}

The $K$-Bessel function has the following representation for $n\in\N_0$ \cite[Subsection 3.71, equation (12)]{W}
\begin{equation}\label{eqn:KBessel_expn}
	K_{n+\frac12} (x) = \sqrt{\frac{\pi}{2x}} e^{-x} \sum_{r=0}^{n} \frac{(n+r)!}{r!(n-r)! (2x)^r}.
\end{equation}

\subsection{Eichler integrals}

For $k\ge1$, we define the (re-normalized) {\it Eichler integral of $E_{2k}$} as
\begin{equation*}
	\mathcal E_{2-2k}(\tau) := \sum_{n\ge1} \frac{q^n}{n^{2k-1}\left(1-q^n\right)} = \sum_{n\ge1} \sigma_{1-2k}(n) q^n.
\end{equation*}
For $k\ge2$, these can be completed to harmonic Maass forms. To state the result, let
\begin{align}\label{E:E*}
	&\mathbb E_{2-2k}^*(\tau) := v^{2k-1} \\
	&+ \frac{2(2k)!}{B_{2k}(4\pi)^{2k-1}}\! \left(\zeta(2k-1)+\mathcal E_{2-2k}(\tau)+\sum_{n\ge1}\sigma_{1-2k}(n)\Gamma^*(2k-1,4\pi nv)q^{-n}\right)\!,\nonumber
\end{align}
where $\Gamma^*(s,y):=\frac{\Gamma(s,y)}{\Gamma(s)}$. By \cite[Theorem 1.2 (2)]{BOW}, we have the following lemma. Note that 
\begin{equation*}
	D^{2k-1}(\mathcal{E}_{2-2k}(\tau))= \frac{B_{2k}}{4k}(1-E_{2k} (\tau)).
\end{equation*}
\begin{lem}\label{k2}
	We have that, for $\ell\in\N$,
	\begin{equation*}
		\mathbb E_{2-2k}(\ell\tau) := \frac{B_{2k}(4\pi)^{2k-1}}{2(2k)!} \mathbb E_{2-2k}^*(\ell\tau)
	\end{equation*}
	is a harmonic Maass form of weight $2-2k$ on $\Gamma_0(\ell)$.
\end{lem}

We next turn to the case $k=1$. Let 
\[
	\mathbb E(\tau) :=  \gamma - \log(2) + \frac\pi6v - \frac{\log(v)}{2} - \frac{6\zeta'(2)}{\pi^2} + \mathcal{E}_0(\tau) + \sum_{n\ge1} \sigma_{-1}(n)\overline q^{n}.
\]
From \cite[Theorem 1.2 (1)]{BOW}, we have the following modularity.  
\begin{lem}\label{k1}
	For $\ell\in\N$, we have that $\mathbb E(\ell\tau)$ is a sesquiharmonic Maass form of weight $0$ on $\Gamma_0(\ell)$.
\end{lem}

\section{Proofs of Theorems \ref{thm:g1} and \ref{thm:g_1 limit}}

In this section, we prove Theorems \ref{thm:g1} and \ref{thm:g_1 limit}. We start with \Cref{thm:g1}.

\begin{proof}[Proof of \Cref{thm:g1}]
	To find a completion for $g_1(q)$, we rewrite
	\begin{equation*}
		g_1(q) = \sum_{n\ge1} \frac{q^n}{n\left(1+q^n\right)} = \sum_{n\ge1} \frac{q^n}{n\left(1-q^n\right)} - 2\sum_{n\ge1} \frac{q^{2n}}{n\left(1-q^{2n}\right)} = \mathcal E_0(\tau) - 2\mathcal E_0(2\tau).
	\end{equation*}

	By \Cref{k1}, $\mathbb E(\tau)-2\mathbb E(2\tau)$ is a sesquiharmonic Maass form of weight $0$ on $\Gamma_0(2)$. A direct calculation shows that
	\begin{equation*}
		\mathbb E(\tau)-2\mathbb E(2\tau)=\widehat g_1(\tau).
	\end{equation*}

	We now apply $\xi_0$ to find
	\begin{equation*}
		\xi_0(\mathbb E(\tau)) = \frac\pi6 - \frac1{2v} - 4\pi\sum_{n\ge1} n\sigma_{-1}(n)q^n.
	\end{equation*}
	Using $n\sigma_{-1}(n)=\sigma_1(n)$, the claim on $\xi_0$ follows.
\end{proof}

Next, we relate $\widehat g_1$ to a Maass Eisenstein series.

\begin{proof}[Proof of \Cref{thm:g_1 limit}]
	In \cite[Lemma 3.1]{BKK}, the authors showed that
	\[
		\sum_{n \ge 1} s_1(n)q^n =  \frac{P(q)}{P\left(q^2\right)} \Log\left(\frac{P(q)}{P^2\left(q^2\right)}\right).
	\]
	By \eqref{g1}, we note that
	\[
		g_1(q) = \Log\left(\frac{P(q)}{P^2\left(q^2\right)}\right).
	\]
	Thus, we arrive at
	\begin{equation*}
		g_1(q) = 2\Log(\eta(2\tau))-\frac{\pi i\tau}4-\Log(\eta(\tau)).
	\end{equation*}
	Using \eqref{Kr} gives the claim.
\end{proof}

\section{Proof of \Cref{thm:gk}}

Throughout this section, we let $k \geq 2$. To prove \Cref{thm:gk}, we first write $g_k(q)$ in terms of Eichler integrals.
\begin{lem}\label{gk}
	We have
	\begin{equation*}
		g_k(q) = \left(q\frac{d}{d q}\right)^{k-1} \left(\mathcal E_{2-2k}(\tau) - 2\mathcal{E}_{2-2k}(2\tau)\right).
	\end{equation*}
\end{lem}
\begin{proof}
	Since
	\begin{equation*}
		w\frac d{dw} \Log(1+w) = \frac w{1-w} - \frac{2w^2}{1-w^2},
	\end{equation*}
	we have
	\begin{equation*}
		\left(w\frac d{dw}\right)^k \Log(1+w) = \left(w\frac d{dw}\right)^{k-1} \left(\frac w{1-w}-\frac{2w^2}{1-w^2}\right).
	\end{equation*}
	This implies that
	\begin{equation*}
		\left[\left(\zeta\frac d{d\zeta}\right)^k\Log\left(1+\zeta^\frac1nq^n\right)\right]_{\zeta=1} = \frac1{n^{2k-1}} \left(q\frac d{dq}\right)^{k-1} \left(\frac{q^n}{1-q^n}-\frac{2q^{2n}}{1-q^{2n}}\right).
	\end{equation*}
	Therefore, we conclude that
	\begin{align*}
		g_k(q) &= \left[\left(\zeta\frac d{d\zeta}\right)^{\!k}\sum_{n\ge1}\Log\!\left(1\!+\!\zeta^\frac1nq^n\right)\right]_{\zeta=1} \!
		= \left(q\frac d{dq}\right)^{\!k-1} \sum_{n \geq 1} \frac1{n^{2k-1}} \left( \frac{q^n}{1\!-\!q^n}-\frac{2q^{2n}}{1\!-\!q^{2n}}\right) \\
		&= \left(q\frac d{dq}\right)^{k-1} \left(\mathcal E_{2-2k}(\tau)-2\mathcal E_{2-2k}(2\tau)\right).\qedhere
	\end{align*}
\end{proof}

Now we define
\begin{equation*}
	g_k^*(\tau) := \left(-\frac1{4\pi}\right)^{k-1} R_{2-2k}^{k-1}\left(\mathbb E_{2-2k}(\tau)-2\mathbb E_{2-2k}(2\tau)\right).
\end{equation*}
Our goal is to show that $g_k^*(\tau)=\widehat g_k(\tau)$. To this end, we first compute the Fourier expansion of $R_{2-2k}^{k-1}(\mathbb E_{2-2k}(\tau))$.
\begin{lem}\label{L:RE}
	We have
	\begin{multline*}
		R_{2-2k}^{k-1} \left(\mathbb E_{2-2k}(\tau)\right) =  \frac{B_{2k}(4\pi)^{2k-1}}{2(2k)!}(k-1)! v^k + (-1)^{k+1} \zeta(2k-1) \frac{(2k-2)!}{(k-1)!} v^{1-k}\\
		+ (-1)^{k+1} \sqrt{\frac{v}{\pi}} \sum_{n\ge1} \sigma_{1-2k}(n) (4\pi n)^{k-\frac12} K_{k-\frac12} (2\pi nv) \left(e^{2\pi inu}+e^{-2\pi inu}\right).
	\end{multline*}
\end{lem}
\begin{proof}
	By \eqref{E:E*} and Lemma \ref{k2}, we have
	\begin{multline*}
		R_{2-2k}^{k-1} \left(\mathbb E_{2-2k}(\tau)\right)  = \frac{B_{2k}(4\pi)^{2k-1}}{2(2k)!} R_{2-2k}^{k-1}\!\left(v^{2k-1}\right) + \zeta(2k-1) R_{2-2k}^{k-1}(1)\\
	 + R_{2-2k}^{k-1}\!\left(\mathcal E_{2-2k}(\tau)\right) + \sum_{n\ge1} \sigma_{1-2k}(n) R_{2-2k}^{k-1}\!\left(\Gamma^*\!\left(2k-1,4\pi nv\right)q^{-n}\right).
	\end{multline*}
	We now compute the individual summands.
	
	Using \eqref{E:RD}, we find that
	\begin{equation*}
		R_{2-2k}^{k-1}\left(v^{2k-1}\right) = \sum_{r=0}^{k-1} (-1)^r \binom{k-1}{r} (2-2k+r)_{k-1-r} v^{r+1-k} (4\pi)^r  D^r\left(v^{2k-1}\right).
	\end{equation*}
	Since
	\begin{align}
		(4\pi)^r D^r\left(v^{2k-1}\right) &= (-1)^r \frac{(2k-1)!}{(2k-r-1)!} v^{2k-1-r}\nonumber\\
		\intertext{and}
		(2-2k+r)_{k-1-r} &= (-1)^{k+r+1} \frac{(2k-r-2)!}{(k-1)!}, \label{E:FA}
	\end{align}
	we obtain that
	\begin{equation*}
		R_{2-2k}^{k-1}\left(v^{2k-1}\right) = (-1)^{k+1} (2k-1)! v^k \sum_{r=0}^{k-1} (-1)^r \frac{1}{r!(k-r-1)!} \frac{1}{2k-r-1}.
	\end{equation*}
	To evaluate the summation above, we let
	\begin{equation*}
		f(x) := \sum_{r=0}^{k-1} \binom{k-1}{r} \frac{(-1)^r x^{2k-r-1}}{2k-r-1}.
	\end{equation*}
	Observe that $f(0)=0$ and
	\begin{equation*}
		f'(x) = (-1)^{k+1} x^{k-1} (1-x)^{k-1}.
	\end{equation*}
	Thus, we have
	\begin{equation*}
		f(x) = (-1)^{k+1} \int_0^x t^{k-1} (1-t)^{k-1}dt.
	\end{equation*}
	In particular,
	\begin{equation*}
		f(1) = (-1)^{k+1} \int_0^1 t^{k-1} (1-t)^{k-1}dt = (-1)^{k+1} \beta(k,k),
	\end{equation*}
	where (for $a,b>0$)
	\begin{equation*}
		\beta(a,b) := \int_0^1 t^{a-1} (1-t)^{b-1} dt
	\end{equation*}
	is the {\it $\beta$-function}. Recall that
	\begin{equation*}
		\beta(a,b) = \frac{\Gamma(a)\Gamma(b)}{\Gamma(a+b)}.
	\end{equation*}
	This gives that
	\begin{equation*}
		f(1) = (-1)^{k+1} \frac{\Gamma(k)\Gamma(k)}{\Gamma(2k)} = \frac{(-1)^{k+1}(k-1)!^2}{(2k-1)!}.
	\end{equation*}
	Hence, we arrive at 
	\begin{equation*}
		R_{2-2k}^{k-1} \left(v^{2k-1}\right) = \frac{(-1)^{k+1}(2k-1)!}{(k-1)!} f(1)v^k = (k-1)!v^k,
	\end{equation*}
	which gives the first claimed summand.

	For the second summand, we find that, using \eqref{E:RD} and \eqref{E:FA},
	\begin{equation*}
		R_{2-2k}^{k-1}(1) = (-1)^{k+1} \frac{ (2k-2)!}{(k-1)!} v^{1-k}.
	\end{equation*}

	Next, since we may write
	\begin{equation*}
		\Gamma^*(2k-1,4\pi nv)q^{-n} = \frac{\overline q^{n}}{\Gamma(2k-1)} \int_0^\infty (t+4\pi nv)^{2k-2} e^{-t} dt,
	\end{equation*}
	we have
	\begin{equation*}
		R_{2-2k}^{k-1}\!\left(\Gamma^*(2k-1,4\pi nv)q^{-n}\right) = \frac{\overline q^{n}}{\Gamma(2k-1)} \int_0^\infty \! R_{2-2k}^{k-1}\!\left((t+4\pi nv)^{2k-2}\right) e^{-t} dt.
	\end{equation*}
	Now, using \eqref{E:RD}, we obtain
	\begin{multline*}
		R_{2-2k}^{k-1}\left((t+4\pi nv)^{2k-2}\right)\\
		= \sum_{r=0}^{k-1} (-1)^r \binom{k-1}{r} (2-2k+r)_{k-1-r} v^{r+1-k} (4\pi)^r D^r\left((t+4\pi nv)^{2k-2}\right).
	\end{multline*}
	By \eqref{E:FA} and the following evaluation
	\begin{equation*}
		D^r\left((t+4\pi nv)^{2k-2}\right) = (-n)^r \frac{(2k-2)!}{(2k-r-2)!} (t+4\pi nv)^{2k-2-r},
	\end{equation*}
	we derive that
	\begin{equation*}
		R_{2-2k}^{k-1} \left((t+4\pi nv)^{2k-2}\right) = (-1)^{k+1} \frac{(2k-2)!}{(k-1)!} v^{1-k}  t^{k-1} (t+4\pi nv)^{k-1}.
	\end{equation*}
	This implies that
	\begin{equation*}
		R_{2-2k}^{k-1}\left(\Gamma^*(2k-1,4\pi nv)q^{-n}\right) = \frac{(-1)^{k+1}v^{1-k} \overline q^{n}}{(k-1)!} \int_0^\infty t^{k-1} (t+4\pi nv)^{k-1} e^{-t} dt.
	\end{equation*}
	We next evaluate
	\begin{align*}
		\int_0^\infty t^{k-1} (t+4\pi nv)^{k-1} e^{-t} dt &= \sum_{r=0}^{k-1} \binom{k-1}{r} (4\pi n v)^{k-1-r} \int_0^\infty t^{k-1+r} e^{-t} dt\\
		&= \sum_{r=0}^{k-1} \binom{k-1}{r} (4\pi n v)^{k-1-r} \Gamma(k\!+\!r)\\
		&= (k\!-\!1)! (4\pi nv)^{k-1} \sum_{r=0}^{k-1} \frac{(k\!+\!r\!-\!1)!}{r!(k\!-\!r\!-\!1)!(4\pi nv)^r}.
	\end{align*}
	Using \eqref{eqn:KBessel_expn}, we arrive at
	\begin{equation*}
		R_{2-2k}^{k-1}\left(\Gamma^*(2k-1,4\pi nv)q^{-n}\right) = (-1)^{k+1}  \sqrt{\frac v{\pi}}  (4\pi n)^{k-\frac12}K_{k-\frac12}(2\pi nv) e^{-2\pi inu},
	\end{equation*}
	which gives the contribution of the second term in the summation over $n$.

	Lastly, we evaluate
	\begin{equation*}
		R_{2-2k}^{k-1}\left(\mathcal E_{2-2k}(\tau)\right) = \sum_{n\ge1} \sigma_{1-2k}(n) R_{2-2k}^{k-1}\left(q^n\right).
	\end{equation*}
	By \eqref{E:RD}, we find
	\begin{equation*}
		R_{2-2k}^{k-1}\left(q^n\right) = \sum_{r=0}^{k-1} (-1)^r \binom{k-1}r (2-2k+r)_{k-1-r} v^{r+1-k} (4\pi)^r  D^r\left(q^n\right).
	\end{equation*}
	Computing
	\begin{equation*}
		D^r\left(q^n\right) = n^r q^n,
	\end{equation*}
	and using \eqref{E:FA}, we obtain
	\begin{align*}
		R_{2-2k}^{k-1}\left(q^n\right) &= (-1)^{k+1} v^{1-k} q^n \sum_{r=0}^{k-1} \frac{(2k-r-2)!}{r!(k-1-r)!} (4\pi nv)^r\\
		&=(-1)^{k+1} (4\pi n)^{k-1} q^n \sum_{r=0}^{k-1} \frac{(k+r-1)!}{r!(k-1-r)! (4\pi n v)^r}
	\end{align*}
	by making the changes of variables $r\mapsto k-1-r$. Using \eqref{eqn:KBessel_expn}, we arrive at
	\begin{equation*}
		R_{2-2k}^{k-1}\left(q^n\right) =  (-1)^{k+1} \sqrt{\frac v {\pi}} (4\pi n)^{k-\frac12}  K_{k-\frac12} (2\pi nv) e^{2\pi inu}.
	\end{equation*}
	Hence, we conclude that
	\[
		R_{2-2k}^{k-1}\left(\mathcal E_{2-2k}(\tau)\right) =  (-1)^{k+1} \sqrt{\frac v {\pi}}  \sum_{n\ge1} \sigma_{1-2k}(n) (4\pi n)^{k-\frac12} K_{k-\frac12} (2\pi nv) e^{2\pi inu},
	\]
	which gives the remaining contribution.
\end{proof}

Our next lemma relates $R_{2-2k}^{k-1}(\mathbb E_{2-2k}(\tau))$ to $E(\tau;k)$. 
\begin{lem}\label{REE}
	We have
	\begin{equation*}
		R_{2-2k}^{k-1}(\mathbb E_{2-2k}(\tau)) = \frac{B_{2k}(4\pi)^{2k-1}(k-1)!}{2(2k)!} E(\tau;k).
	\end{equation*}
\end{lem}
\begin{proof}
	From \Cref{L:BFOR}, we obtain
	\begin{multline}\label{eqn:cMES}
		 \frac{B_{2k}(4\pi)^{2k-1}(k-1)!}{2(2k)!} E(\tau; k)
		 =  \frac{B_{2k}(4\pi)^{2k-1}(k-1)!}{2(2k)!}\\ \times\vast(v^k + \phi(k) v^{1-k} + 2\sqrt v \sum_{n\in\Z\setminus\{0\}} \phi(n,k) K_{k-\frac12} (2\pi |n| v) e^{2\pi inu}\vast).
	\end{multline}

	We now compare \eqref{eqn:cMES} and \Cref{L:RE} term-wise. One directly sees that the constants in front of $v^k$ are equal. Using
	\begin{equation*}
		B_{2k} = \frac{2(-1)^{k+1}(2k)!}{(2\pi)^{2k}} \zeta(2k) \qquad \text{for } k \geq 1
	\end{equation*}
	and
	\begin{equation*}
		\Gamma(k-1) \Gamma\left(k-\frac12\right) = 2^{3-2k} \sqrt\pi \Gamma(2k-2),
	\end{equation*}
	we also observe that the coefficients in front of $v^{1-k}$ coincide.

    Using
	\begin{equation}\label{eq:sigma}
		\sum_{ab=n} \pa ab ^{k-\frac12} = n^{k-\frac12} \sigma_{1-2k}(n)\qedhere
	\end{equation}
	the claim follows for the remaining term.
	\end{proof}

Next, we compute the holomorphic part of $(-\frac1{4\pi})^{k-1} R_{2-2k}^{k-1}(\mathbb E_{2-2k}(\tau))$.

\begin{lem}\label{hol}
	The holomorphic part of
	\begin{equation*}
		\left(-\frac1{4\pi}\right)^{k-1} R_{2-2k}^{k-1}(\mathbb E_{2-2k}(\tau))
	\end{equation*}
	is
	\begin{equation*}
		\left(q\frac d{dq}\right)^{k-1} \mathcal E_{2-2k}(\tau).
	\end{equation*}
\end{lem}
\begin{proof}
	Clearly the terms with $v^k$ and $v^{1-k}$ do not contribute to the holomorphic part. To finish the claim, we need to study
	\begin{equation*}
		f_\pm(\tau) := \sqrt v K_{k-\frac12}(v) e^{\pm iu}.
	\end{equation*}
	From \eqref{eqn:KBessel_expn}, we obtain
	\[
		\sqrt v K_{k-\frac12}(v) e^v =  \sqrt{\frac{\pi}{2}} \sum_{r=0}^{k-1} \frac{(k-1+r)!}{r!(k-1-r)!(2 v)^r}
	\]
	and note that the constant term as a function of $\frac1v$ is $\sqrt\frac{\pi}2$.
	Since
	\begin{equation*}
			e^{-v+iu} = e^{i(u+iv)} = e^{i\tau} \quad\text{and}\quad
			e^{-v-iu} = e^{-i(u-iv)} = e^{-i\overline\tau},
	\end{equation*}
	the holomorphic part of $f_+(\tau)$ is $\sqrt{\frac\pi2} e^{i\tau}$ and $f_-(\tau)$ does not have a holomorphic part. Hence, the holomorphic part of $(-\frac1{4\pi})^{k-1} R_{2-2k}^{k-1}(\mathbb E_{2-2k}(\tau))$ is as claimed.
\end{proof}

We are now ready to prove \Cref{thm:gk}.
\begin{proof}[Proof of \Cref{thm:gk}]
	Since
	\begin{equation*}
		R_\kappa(f(r\tau)) = r\left[R_\kappa\left(f(\tau)\right)\right]_{\tau\mapsto r\tau},
	\end{equation*}
	we have
	\begin{equation*}
		R_{2-2k}^{k-1}\left(\mathbb E_{2-2k}(2\tau)\right) = 2^{k-1} \left[ R_{2-2k}^{k-1}\left(\mathbb E_{2-2k}(\tau)\right)\right]_{\tau\mapsto 2\tau}.
	\end{equation*}
	By Lemmas \ref{REE} and \ref{hol}, $(q\frac{d}{dq})^{k-1}\mathcal E_{2-2k}(\tau)$ is the holomorphic part of
	\begin{equation*}
		\left(-\frac1{4\pi}\right)^{k-1} R_{2-2k}^{k-1}(\mathbb E_{2-2k}(\tau)) = \frac{(-1)^{k+1} (4\pi)^k(k-1)!B_{2k}}{2(2k)!}E(\tau;k),
	\end{equation*}
	Therefore, by \Cref{gk}, we conclude that $g_k(q)$ is the holomorphic part of $\widehat g_k(\tau)$.

	We next determine the modularity properties of $\widehat g_k$. We have that $E(\tau;k)$ is a Maass form on $\SL_2(\Z)$, so $E(2\tau;k)$ is a Maass form on $\Gamma_0(2)$. The eigenvalue of $E(\tau;k)$ is $k(1-k)$ so the same holds for $\widehat g_k(\tau)$.
\end{proof}

\section{An example}
Here we explicitly work out the example of $g_2(q)$, writing it in terms of more elementary functions.

By Theorem \ref{thm:gk}, for $k=2$, we have that $g_2(q)$ is the holomorphic part of
\begin{equation*}
	\widehat g_2(\tau) =  \frac{\pi^2}{90} \left( E(\tau;2)-4 E(2\tau;2) \right).
\end{equation*}
We have, by \Cref{L:BFOR},
\begin{equation*}
	E(\tau;2) = v^2 + \frac{\phi(2)}v + 2\sqrt v\sum_{n\in\Z\setminus\{0\}} \phi(n,2) K_\frac32(2\pi|n|v) e^{2\pi inu}.
\end{equation*}
We then compute
\[
	\phi(2) = \frac{45}{\pi^3}\zeta(3), \quad K_\frac32(x) = \sqrt{\frac\pi{2x}} \left(1+\frac1x\right) e^{-x}.
\]	
We also have that by \eqref{eq:sigma},
\[
	\phi(n,2) = \frac{90}{\pi^2} \sum_{ab=|n|} \left(\frac ab\right)^\frac32 = \frac{90}{\pi^2} |n|^\frac32 \sigma_{-3}(|n|).
\]
Thus, we obtain that
\begin{equation*}
	E(\tau;2) = v^2 + \frac{45\zeta(3)}{\pi^3v} + \frac{90}{\pi^2} \sum_{n\ge1} n\sigma_{-3} (n) \left(1+\frac{1}{2\pi nv}\right) \left(q^n + \overline q^n \right),
\end{equation*}
which gives
\begin{align*}
	\widehat g_2(\tau) &= -\frac{\pi^2}6 v^2 - \frac{\zeta(3)}{2\pi v} + \sum_{n\ge1} n \sigma_{-3}(n) \left(q^n-4q^{2n}+\overline q^n-4\overline q^{2n}\right)\\
	&\hspace{5cm}+\frac{1}{2\pi v} \sum_{n\ge1} \sigma_{-3}(n) \left(q^n - 2q^{2n} +\overline q^n - {2} \overline{q}^{2n}  \right).
\end{align*}
Now we observe that from \Cref{gk},
\begin{equation*}
	\sum_{n\ge1} n\sigma_{-3}(n)\left(q^n-4q^{2n}\right)  = g_2(q).
\end{equation*}
Therefore, we conclude that
\[
	\widehat g_2(\tau) = g_2(q) + g_2(\overline q) - \frac{\pi^2}6 v^2 + \frac{1}{2\pi v} \bigg( \mathcal G_2 (q)+ \mathcal G_2 (\overline{q}) - \zeta(3) \bigg),
\]
where
\[
	\mathcal G_2 (q):= \mathcal{E}_{-2} (\tau) - 2 \mathcal{E}_{-2} (2\tau).
\]

Note that $(-q;q)_\infty \mathcal G_2(\tau)$ can be interpreted as the generating function for $s_{3}^* (n) = \sum_{\lambda \in \mathcal{D}_n} {\rm srp}_3 (\lambda)$, where ${\rm srp}_k (\lambda) = \sum_{j=1}^{\ell (\lambda)} \frac{1}{\lambda_j^k}$.  We first observe that
\begin{equation*}
	\mathcal G_2(q) = \sum_{n\ge1} \frac{q^n}{n^3\left(1-q^n\right)} - 2\sum_{n\ge1} \frac{q^{2n}}{n^3\left(1-q^{2n}\right)} = \sum_{n\ge1} \frac{q^n}{n^3\left(1+q^n\right)}.
\end{equation*}
From the theory of partitions, it is not hard to see that
\[
	f_3(\zeta;q) := \sum_{\lambda \in \mathcal{D}_n} \zeta^{{\rm srp}_3 (\lambda)} q^{|\lambda|} =  \prod_{n\ge1} \left(1+\zeta^\frac1{n^3} q^n\right).
\]
Therefore, we find that
\begin{equation*}
	\sum_{n \geq 0} s_3^*(n) q^n = \left[  \dfrac{\partial}{\partial\zeta} f_3(\zeta;q) \right]_{\zeta=1} = (-q;q)_{\infty} \sum_{n \geq 1} \frac{q^n}{n^3(1+q^n)} = (-q;q)_\infty \mathcal G_2 (q).
\end{equation*}

\section{Questions for Future research}
\begin{enumerate}[leftmargin=*]
	\item It is natural to ask how ${\rm srp}(\lambda)$ is distributed along the set $\mathcal{D}_n$. Recently, Bridges \cite{WB} proved that for $x\in\R$
	\[
		\lim_{n \to \infty} P_n \left( 2S - \log \left(\sqrt{3n}\right) \le x \right) = P\left( H \le x \right),
	\]
	where $S(\lambda) := \sum_{j=1}^n \frac1{\lambda_j}$,  $\lambda \in \mathcal{D}_n$ is chosen uniform randomly, and $H$ is a random harmonic sum, i.e., $H := \sum_{k\ge1} \frac{\varepsilon_k}{k}$. Here, $\varepsilon_k$ are independent random variables with $P(\varepsilon_k \!=\! \pm 1) = \frac{1}{2}$. Note that from \cite[Theorem 1.1]{KK}, $\mathbb{E} (2S) = \log(\sqrt{3n})$.   The asymptotics of the moments of  ${\rm srp}(\lambda)$, i.e., $s_k (n)$,  give more detailed information on how ${\rm srp}(\lambda)$ is distributed. The modularity of $\widehat{g}_k$ established in this paper could be helpful to derive sharper asymptotics of $s_k (n)$.  
	
	\item It would be interesting to investigate analogues of ${\rm srp}(\lambda)$ which can be related with modular forms like Eisenstein series. As a possible candidate, one may consider a twisted sum of reciprocals of parts in the partition $\lambda$, 
	$${\rm srp}_{\chi_p} (\lambda) = \sum_{j=1}^{\ell(\lambda)} \frac{\left(\frac{\lambda_j}{p}\right)}{\lambda_j}$$ 
	for a partition $\lambda \in \mathcal{D}_n$, where $(\frac{\cdot}{p})$ is the Legendre symbol (mod $p$) for an odd prime $p$. Now we define the first moment of ${\rm srp}_{\chi_p} (\lambda)$ as
	\[
		s_{\chi_p} (n) := \sum_{\lambda \in \mathcal{D}_n} {\rm srp}_{\chi_p} (\lambda).
	\]
	Then, the generating function for $s_{\chi_p} (n)$ is
	\begin{equation*}
		T_{\chi_p} (q) := (-q;q)_{\infty}  \sum_{m \geq 1} \frac{\chi_p (m)q^m}{m \!\left(1\!+\!q^m\right)} 
		= -(-q;q)_{\infty}  \sum_{n \geq 1} \frac{1}{n} \sum_{d|n} (-1)^d \chi_p \!\left(\frac n d \right)\! d q^n. \hspace{-.6cm}
	\end{equation*}
	Note that the summation part is a formally anti-derivative (w.r.t. $q \frac{d}{dq}$) of 
	\[
		\sum_{n \geq 1} \sum_{d|n} (-1)^d \chi_p \left(\frac n d \right) d q^n.
	\]
	This hints possible modular properties of $T_{\chi_p} (q)$ and higher moments for ${\rm srp}_{\chi_p}(\lambda)$. 
	
	\item Eichler integrals occur in many places. For example, by the works of Han and Ji \cite{Ha, HJ}, they are related to $q$-brackets. So it would be interesting to investigate whether our functions are related to these. 
\end{enumerate}


\end{document}